\documentclass[11pt]{amsart}
\usepackage{xy, graphicx, color, hyperref}
\xyoption{all}
\usepackage{listings}
\lstdefinelanguage{Sage}[]{Python}
{morekeywords={False,sage,True},sensitive=true}
\definecolor{dblackcolor}{rgb}{0.0,0.0,0.0}
\definecolor{dbluecolor}{rgb}{0.01,0.02,0.7}
\definecolor{dgreencolor}{rgb}{0.2,0.4,0.0}
\definecolor{dgraycolor}{rgb}{0.30,0.3,0.30}

\newcommand{\la}{\lambda}
\newcommand{\half}{\frac 12}
\newcommand{\C}{\mathbf C}
\newcommand{\core}[2]{\mathrm{core}_{#2}{#1}}
\newcommand{\quo}[2]{\mathrm{quo}_{#2}{#1}}

\DeclareMathOperator{\bin}{bin}

\newtheorem{theorem}{Theorem}
\newtheorem{lemma}[theorem]{Lemma}

\newtheorem{corollary}[theorem]{Corollary}
\theoremstyle{example}

\theoremstyle{remark}
\newtheorem*{remark*}{Remark}
\title[Representations with non-trivial determinant]{Representations of symmetric groups\\with non-trivial determinant}
\author{Arvind Ayyer}
\address{AA: Department of Mathematics, Indian Institute of Science, Bengaluru 560012, India.}
\email{arvind@math.iisc.ernet.in}
\author{Amritanshu Prasad}
\address{AP: The Institute of Mathematical Sciences, CIT campus, Taramani Chennai 600113, India.}
\email{amri@imsc.res.in}
\author{Steven Spallone}
\address{SS: Indian Institute of Science Education and Research, Pashan, Pune 411008, India.}
\email{sspallone@iiserpune.ac.in}
\keywords{Symmetric group, irreducible representations, permutation representations, determinants, core, quotients, core-towers, Bell numbers}
\subjclass[2010]{05E10, 20C30, 05A17, 05A15}

\begin{document}
\maketitle
\begin{abstract}
  We give a closed formula for the number of partitions $\lambda$ of $n$ such that the corresponding irreducible representation $V_\lambda$ of $S_n$ has non-trivial determinant.
  We determine how many of these partitions are self-conjugate and how many are hooks.
  This is achieved by characterizing the $2$-core towers of such partitions.
  We also obtain a formula for the number of partitions of $n$ such that the associated permutation representation of $S_n$ has non-trivial determinant.
\end{abstract}

\section{Introduction}
\label{sec:introduction}
Let $(\rho, V)$ be a complex representation of the symmetric group $S_n$.
Let $\det:GL_\C(V)\to \C^*$ denote the determinant function.
The composition $\det\circ\rho:S_n\to \C^*$, being a multiplicative character of $S_n$, is either the trivial character or the sign character.
We call $(\rho, V)$ a \emph{chiral representation} if $\det\circ\rho$ is the sign character of $S_n$.

Recall that the irreducible complex representations of $S_n$ are parametrized by partitions of $n$.
For brevity, we shall say that $\lambda$ is a \emph{chiral partition} if the corresponding representation $(\rho_\lambda, V_\lambda)$ is chiral\footnote{The use of the field $\C$ is not important. Our results hold over any field of characteristic different from two, so long as $V_\lambda$ is interpreted as the representation corresponding to $\lambda$ in Young's seminormal form.}.

It is natural to ask, given a positive integer $n$, how many partitions $\lambda$ of $n$ are chiral.
According to Stanley \cite[Exercise~7.55]{ec2}, this problem was first considered by L. Solomon.
We review his contribution, which is explained by Stanley in the exercises in his book, in Section~\ref{sec:2-core-towers-chiral}.
Let $b(n)$ denote the number of chiral partitions of $n$.
The complementary sequence $p(n) - b(n)$ ($p(n)$ denotes the number of partitions of $n$) was added to the Online Encyclopedia of Integer Sequences \cite{Sloane} as sequence \texttt{A045923} in 1999.
It looks mysterious:
\begin{displaymath}
  1, 1, 1, 2, 2, 7, 7, 10, 10, 34, 40, 53, 61, 103, 112, 143, 145, 369, 458, 579,\dotsc,
\end{displaymath}
but the sequence $b(n)$ itself looks much nicer:
\begin{displaymath}
  0, 1, 2, 3, 5, 4, 8, 12, 20, 8, 16, 24, 40, 32, 64, 88, 152, 16, 32, 48,\dotsc.
\end{displaymath}
The most striking result of this paper is a formula for $b(n)$:
\begin{theorem}
  \label{theorem:all}
  If $n$ is an integer having binary expansion
  \begin{equation}
    \label{eq:1}
    n = \epsilon + 2^{k_1} + 2^{k_2} + \dotsb + 2^{k_r},\, \epsilon \in \{0, 1\},\; 0<k_1<k_2<\dotsb <k_r,
  \end{equation}
  then the number of chiral partitions of $n$ is given by
  \begin{displaymath}
    b(n) = 2^{k_2+\dotsb + k_r}\Big(2^{k_1-1} + \sum_{v=1}^{k_1-1} 2^{(v+1)(k_1-2) - \binom v2} + \epsilon 2^{\binom{k_1}2}\Big).
  \end{displaymath}
\end{theorem}
Given a partition $\lambda$, let $f_\lambda$ denote the dimension of $V_\lambda$.
For any integer $m$, let $v_2(m)$ denote the largest among integers $v$ such that $2^v$ divides $m$.
Theorem~1 is a consequence of a more refined counting result:
\begin{theorem}
  \label{theorem:refined}
  If $n$ is an integer having binary expansion (\ref{eq:1}),
  then the number $b_v(n)$ of chiral partitions $\lambda$ of $n$ for which $v_2(f_\lambda) = v$ is given by
  \begin{displaymath}
    b_v(n) = 2^{k_2+\dotsb+k_r}\times
    \begin{cases}
      2^{k_1-1} & \text{if } v =0,\\
      2^{(v+1)(k_1-2)-\binom v2} &\text{if } 0<v<k_1,\\
      \epsilon 2^{\binom{k_1}2} &\text{if } v = k,\\
      0 &\text{if } v> k_1.
    \end{cases}
  \end{displaymath}
\end{theorem}
Theorem~\ref{theorem:refined} follows from the characterization of the $2$-core tower of a chiral partition in Theorem~\ref{theorem:chiral-towers}.
A characterization of $2$-core towers of self-conjugate chiral partitions (see Corollary~\ref{corollary:self-conj} and its proof) gives us: 
\begin{theorem}
  \label{theorem:self-conj-count}
  A positive integer $n$ admits a self-conjugate chiral partition if and only if $n=3$ or $n=2^k+\epsilon$ for some $k\geq 2$ and $\epsilon\in\{0,1\}$.
  There is one self-conjugate chiral partition of $3$.
  When $n = 2^k + \epsilon$ with $k\geq 2$ and $\epsilon\in \{0,1\}$, the number of self-conjugate chiral partitions is $2^{k-2}$.
\end{theorem}
A study of the number $a(n)$ of odd-dimensional representations of $S_n$ was carried out by Macdonald in \cite{macdonald1971degrees}.
In Section~\ref{sec:growth-bn}, we show that $2/5\leq a(n)/b(n)\leq 1$.
It turns out that the growth of $b(n)$ is much slower than that of the partition function.
Our characterization of chiral partitions allows for rapid enumeration and uniform random sampling of the very sparse subset of chiral partitions.

Using only elementary results about the parity of binomial coefficients, we count the number of chiral hooks of size $n$ in Section~\ref{sec:chiral-hooks}.

Another important class of representations of $S_n$ is the class of permutation representations.
For each partition $\lambda = (\lambda_1,\ldots,\lambda_l)$ of $n$ let
\begin{displaymath}
  X_\lambda = \{(X_1,\ldots, X_l)\mid X_1\sqcup\cdots\sqcup X_l = \{1, \ldots, n\},\, |X_i| = \lambda_i\},
\end{displaymath}
the set of all set-partitions of $\{1,\dotsc,n\}$ of shape $\lambda$.
The action of $S_n$ on $\{1,\dotsc,n\}$ gives rise to an action on $X_\lambda$.

The space $\C[X_\lambda]$ of $\C$-valued functions on $X_\lambda$ becomes a representation of $S_n$ whose dimension is the multinomial coefficient $\binom n{\lambda_1,\dotsc,\lambda_l}$.
The characters of $\C[X_\lambda]$ form a basis of the space of class functions on $S_n$.   
In Section~\ref{sec:perm-reps}, we characterize the partitions $\lambda$ of $n$ for which $\C[X_\lambda]$ is chiral (see Theorem~\ref{perm.chir.descr}).
This characterization allows us to compute the number $c(n)$ of such partitions exactly.
Let $\nu(n)$ denote the number of times $1$ occurs in the binary expansion of $n$ (for $n$ as in (\ref{eq:1}), $\nu(n) = r+\epsilon$).
Let $B_k$ denote the $k$th Bell number, which is the number of set-partitions of $\{1, 2, \ldots, k \}$.
\begin{theorem}
  \label{theorem:perm-rep}  Let $n \geq 3$. 
  If $n$ is even, then
  \begin{equation}
    \label{eq:14}
    c(n) = \half ( B_{\nu(n-2)+2}-B_{\nu(n-2)+1}+B_{\nu(n-2)}).
  \end{equation}
  If $n$ is odd, with $v_2(n-1) = k$, then
  \begin{multline}
    \label{eq:15}
    c(n) = \frac{1}{6}(B_{\nu(n-3)+3} - 3B_{\nu(n-3)+3} + 5B_{\nu(n-3)+1} + 2B_{\nu(n-3)})\\
    +B_{\nu(n)+k-2}+B_{\nu(n)}-2B_{\nu(n)-1}.
  \end{multline}
\end{theorem}

\section{$2$-core Tower of a Partition}
\label{sec:core-tower}
We briefly recall the definition and basic properties of $2$-core towers.
More details can be found in Olsson's monograph \cite{olsson}.

For a partition $\lambda$ and an integer $p>1$, let $\core\lambda p$ denote the $p$-core, and $\quo\lambda p$ denote the $p$-quotient of $\lambda$.
The partition $\core\lambda p$ is what remains of Young diagram of $\lambda$ after successively removing the rims of as many $p$-hooks as possible.
The $p$-quotient $\quo\lambda p$ is a $p$-tuple $(\lambda_0,\dotsc,\lambda_{p-1})$ of partitions.
The total number of cells in $\quo\lambda p$ is the number of $p$-hooks whose rims were removed from $\lambda$ to obtain $\core\lambda p$.
Consequently,
\begin{equation}
  \label{eq:2}
  |\lambda| = |\core\lambda p| + p(|\lambda_0|+\dotsb + |\lambda_{p-1}|).
\end{equation}
The size of the partition $\lambda_k$ in the $p$-quotient is the number of nodes in the Young diagram of $\lambda$ whose hook-lengths are multiples of $p$, and whose hand-nodes have content congruent to $k$ modulo $p$ (by definition, the content of the node $(i,j)$ is $j-i$).
The partition $\lambda$ can be recovered uniquely from $\core\lambda p$ and $\quo\lambda p$.

Given a partition $\lambda$, its $2$-core tower is defined as follows: it has rows numbered by integers $0,1,2,\dotsc$.
The $i$th row of this tower has $2^i$ many $2$-cores.
The $0$th row has the partition $\alpha_\emptyset := \core\lambda{2}$.
The first row consists of the partitions 
\begin{displaymath}
  \alpha_0, \alpha_1,
\end{displaymath}
where, if $\quo\lambda 2 = (\lambda_0,\lambda_1)$, then $\alpha_i = \core{\lambda_i}2$.
Let $\quo{\lambda_i}2 = (\lambda_{i0},\lambda_{i1})$, and define $\alpha_{ij} = \core{\lambda_{ij}}2$.
The second row of the $2$-core tower is 
\begin{displaymath}
  \alpha_{00}, \alpha_{01}, \alpha_{10}, \alpha_{11}.
\end{displaymath}
Inductively, having defined partitions $\lambda_x$ for a binary sequence $x$, define the partitions $\lambda_{x0}$ and $\lambda_{x1}$ by
\begin{equation}
  \label{eq:6}
  \quo{\lambda_x}2 = (\lambda_{x0}, \lambda_{x1}),
\end{equation}
and let $\alpha_{x\epsilon} = \core{\lambda_{x\epsilon}}2$ for $\epsilon=0,1$.
The $i$th row of the $2$-core tower of $\lambda$ consists of the partitions $\alpha_x$, where $x$ runs over the set of all $2^i$ binary sequences of length $i$, listed from left to right in lexicographic order.

Visually, the $2$-core tower is represented as the binary tree:
\begin{displaymath}
  \xymatrix@C=0em{
    &&&&&&& \alpha_{\emptyset}\ar@{-}[dllll]\ar@{-}[drrrr] &&&&&&&\\
    &&&\alpha_{0}\ar@{-}[dll]\ar@{-}[drr] &&&&&&&& \alpha_{1} \ar@{-}[dll]\ar@{-}[drr] &&&\\
    &\alpha_{00}\ar@{-}[dl]\ar@{-}[dr]&&&& \alpha_{01}\ar@{-}[dl]\ar@{-}[dr] &&&&\alpha_{10}\ar@{-}[dl]\ar@{-}[dr]&&&&\alpha_{11}\ar@{-}[dl]\ar@{-}[dr]&\\
    {\begin{matrix}\alpha_{000}\\\vdots\end{matrix}}&&
    {\begin{matrix}\alpha_{001}\\\vdots\end{matrix}}&&
    {\begin{matrix}\alpha_{010}\\\vdots\end{matrix}}&&
    {\begin{matrix}\alpha_{011}\\\vdots\end{matrix}}&&
    {\begin{matrix}\alpha_{100}\\\vdots\end{matrix}}&&
    {\begin{matrix}\alpha_{101}\\\vdots\end{matrix}}&&
    {\begin{matrix}\alpha_{110}\\\vdots\end{matrix}}&&
    {\begin{matrix}\alpha_{111}\\\vdots\end{matrix}}&&
  }
\end{displaymath}
A partition is uniquely determined by its $2$-core tower, which has non-empty partitions in finitely many places.
For example, $(5, 4, 2, 2, 1, 1)$ has $2$-core $(1)$ and $2$-quotient $((2, 2, 1, 1), (1))$.
The partition $(2, 2, 1,1)$ has trivial $2$-core and $2$-quotient $((1,1), (1))$.
Finally, $(1,1)$ has empty $2$-core and quotient $((1),\emptyset)$.
Therefore the $2$-core tower of $(5, 4, 2, 2, 1, 1)$ is given by:
\begin{displaymath}
  \xymatrix@C=0.7em{
    &&&&&&& (1)\ar@{-}[dllll]\ar@{-}[drrrr] &&&&&&&\\
    &&&\emptyset\ar@{-}[dll]\ar@{-}[drr] &&&&&&&& (1) \ar@{-}[dll]\ar@{-}[drr] &&&\\
    &\emptyset\ar@{-}[dl]\ar@{-}[dr]&&&& (1)\ar@{-}[dl]\ar@{-}[dr] &&&&\emptyset\ar@{-}[dl]\ar@{-}[dr]&&&&\emptyset\ar@{-}[dl]\ar@{-}[dr]&\\
    (1)&&\emptyset&&\emptyset&&\emptyset&&\emptyset&&\emptyset&&\emptyset&&\emptyset
  }
\end{displaymath}

Given a partition $\lambda$ of $n$, let $w_i(\lambda)$ denote the sum of the sizes of the partitions in the $i$th row of the $2$-core tower of $\lambda$.
The identity (\ref{eq:2}) implies that
\begin{displaymath}
  n = \sum_{i=0}^\infty w_i(\lambda)2^i.
\end{displaymath}
In particular, the number of non-trivial rows in the $2$-core tower of $\lambda$ is at most the number of digits in the binary expansion of $n$.
Also let $w(\lambda) = \sum_{i\geq 0} w_i(\lambda)$.
Define the \emph{$2$-deviation} of $\lambda$ as
\begin{displaymath}
  e_2(\lambda) = w(\lambda)-\nu(|\lambda|).
\end{displaymath}
The following result gives a formula for the $2$-adic valuation $v_2(f_\lambda)$.
\begin{theorem}
  [{\cite[Proposition~6.4]{olsson}}]
  \label{theorem:oddness}
  For any partition $\lambda$, $v_2(f_\lambda) = e_2(\lambda)$.
\end{theorem}
In the running example $\lambda = (5, 4, 2, 2, 1, 1)$, which is a partition of $15$, the $2$-deviation is zero, and so $f_\lambda$ is odd (in fact, $f_{(5, 4, 2, 2, 1, 1)} = 243243$).

Theorem~\ref{theorem:oddness} characterizes partitions parametrizing representations with dimensions of specified $2$-adic valuation in terms of their $2$-core towers.
Theorem~\ref{theorem:refined} follows from a similar characterization of chiral partitions which we are now ready to state:
\begin{theorem}
  \label{theorem:chiral-towers}
  Suppose $n$ is a positive integer with binary expansion (\ref{eq:1}).
  Then a partition $\lambda$ of $n$ is chiral if and only if one of the following happens:
  \begin{enumerate}
  \item \label{item:6} The partition $\lambda$ satisfies
    \begin{displaymath}
      w_i(\lambda) =
      \begin{cases}
        1 &\text{if } i\in\{k_1,\dotsc,k_r\}, \text{ or  if } \epsilon = 1\text{ and } i =0,\\
        0 &\text{otherwise,}
      \end{cases}
    \end{displaymath}
    and the unique non-trivial partition in the $k_1$th row of the $2$-core tower of $\lambda$ is $\alpha_x$, where the binary sequence $x$ of length $k$ begins with $\epsilon$.
    In this case $f_\lambda$ is odd.
  \item \label{item:7} For some $0<v<k_1$,
    \begin{displaymath}
      w_i(\lambda) =
      \begin{cases}
        2 & \text{if } i = k_1-v,\\
        1 & \text{if } k_1-v+1\leq i\leq k_1-1\text{ or }i\in \{k_2,\dotsc,k_r\},\\
        & \text{or if } \epsilon = 1 \text{ and } i=0,\\
        0 & \text{otherwise,}
      \end{cases}
    \end{displaymath}
    and the two non-trivial partitions in the $(k-v)$th row of the $2$-core tower of $\lambda$ are $\alpha_x$ and $\alpha_y$, for binary sequences $x$ and $y$ such that $x$ begins with $0$ and $y$ begins with $1$.
    In this case $v_2(f_\lambda) = v$.
  \item \label{item:8} We have $\epsilon = 1$ and the partition $\lambda$ satisfies
    \begin{displaymath}
      w_i(\lambda) =
      \begin{cases}
        3 & \text{if } i=0,\\
        1 & \text{if } i\in \{1,\dotsc, k_1-1, k_2,\dotsc, k_r\}.
      \end{cases}
    \end{displaymath}
    In this case, $v_2(f_\lambda) = k_1$.
  \end{enumerate}
\end{theorem}
Not only does Theorem~\ref{theorem:chiral-towers} prove the enumerative results in Theorems~\ref{theorem:all} and~\ref{theorem:refined}, it also provides fast algorithms to:
\begin{itemize}
\item sequentially enumerate all chiral partitions of $n$ with given $v_2(f_\lambda)$.
\item generate a uniformly random chiral partition of $n$ with given $v_2(f_\lambda)$.
\end{itemize}
It also provides the distribution of $v_2(f_\lambda)$ among the chiral partitions of $n$.
In particular, note that if $\lambda$ is a chiral partition of $n$, then $v_2(f_\lambda)\in [0, k_1-1]$ when $n$ is of the form (\ref{eq:1}) with $\epsilon =0$, and $v_2(f_\lambda)\in [0, k_1]$ when $\epsilon = 1$.

Another interesting consequence of Theorem~\ref{theorem:chiral-towers} is a characterization of self-conjugate chiral partitions:
\begin{corollary}
  \label{corollary:self-conj}
  A positive integer $n$ admits a self-conjugate chiral partition if and only if $n=3$, or $n = 2^k + \epsilon$ for some $k\geq 2$ and $\epsilon \in \{0, 1\}$.
  Moreover, $\lambda$ is a self-conjugate chiral partition of $2^k+\epsilon$ if and only if $\lambda$ is self-conjugate and $v_2(f_\lambda) = 1$. 
  The number of self-conjugate chiral partitions of $2^k+\epsilon$ is $2^{k-2}$ for $k\geq 2$.
\end{corollary}
\begin{proof}
  Let $\lambda'$ denote the conjugate of a partition $\lambda$.
  Recall that $\core{(\lambda')}2 = (\core\lambda 2)'$, and if $\quo\lambda 2=(\lambda_0,\lambda_1)$, then $\quo{\lambda'}{2}= (\lambda_1', \lambda_0')$ (see, for example, \cite[Proposition~3.5]{olsson}).
  It follows that the $2$-core tower of $\lambda'$ is obtained by reflecting the $2$-core tower of $\lambda$ about the vertical axis, and then replacing each entry with its conjugate.
  If $\lambda$ is self-conjugate, then its $2$-core tower has to be invariant under this operation.
  In particular, no row numbered by a positive integer can have weight equal to $1$.

  If $n>3$, this can only happen if the case (\ref{item:7}) in Theorem~\ref{theorem:chiral-towers} is realized with $n=2^k+\epsilon$ and $v=1$.

  Now suppose $\lambda$ is a self-conjugate partition of $2^k+\epsilon$ and $v_2(f_\lambda) = 1$.
  By Theorem~\ref{theorem:oddness}, $w_{k-1}(\lambda) = 2$, $w_i(\lambda) = 0$ for $0<i<k-1$, and $w_0(\lambda) = \epsilon$.
  Let $x$ and $y$ be the two binary sequences of length $k-1$ which index the non-empty entries in the $2$-core tower of $\lambda$.
  The symmetry of the $2$-core tower ensures that one begins with $0$ and the other with $1$.
  Thus $\lambda$ satisfies the conditions in (\ref{item:7}), and is therefore chiral.

  In order to count the number of such partitions, note that each of $x$ and $y$ determine each other, and for $x$ beginning with $0$, there are $2^{k-2}$ choices.
\end{proof}
The proof of Theorem~\ref{theorem:chiral-towers} requires the characterization of chiral partitions in terms of counting a class of Young tableaux, which we take up in the next section.

\section{Counting Tableaux}
\label{sec:2-core-towers-chiral}

In Young's seminormal form, the representation $V_\lambda$ has basis indexed by standard tableaux of shape $\lambda$:
\begin{displaymath}
  \{v_T\mid \text{$T$ is a standard tableau of shape $\lambda$}\}.
\end{displaymath}
Let $s_1$ denote the simple transposition $(1,2)$.
The vectors $v_T$ are all eigenvectors for the involution $\rho_\lambda(s_1)$.
Note that, in a standard tableau, $2$ always occurs either in the first row, or in the first column.
We have:
\begin{equation}
  \label{eq:4}
  \rho_\lambda(s_1)v_T =
  \begin{cases}
    v_T & \text{if $2$ lies in the first row of $T$,}\\
    -v_T & \text{if $2$ lies in the first column of $T$.}
  \end{cases}
\end{equation}
Let $g_\lambda$ denote the number of standard tableaux of shape $\lambda$ such that $2$ occurs in the first column of $T$.
The equation (\ref{eq:4}) implies that
\begin{displaymath}
  \det(\rho_\lambda(s_1)) = (-1)^{g_\lambda},
\end{displaymath}
so $\lambda$ is chiral if and only if $g_\lambda$ is odd.
Observe that the trace of $\rho_\lambda(s_1)$ is the difference between the multiplicity of $+1$ as an eigenvalue minus the multiplicity of $-1$ as an eigenvalue.
Since the multiplicity of $+1$ is $f_\lambda - g_\lambda$, and the multiplicity of $-1$ is $g_\lambda$,
\begin{displaymath}
  \mathrm{trace}(\rho_\lambda(s_1)) = f_\lambda - 2g_\lambda.
\end{displaymath}
This is the value of the symmetric group character $\chi_\lambda$ at the class consisting of permutations with cycle decomposition $(2, 1^{n-2})$.
This character value is given by (see \cite[Example~3, p.~11]{macd} or  \cite[Exercise~7.51]{ec2}):
\begin{displaymath}
  \chi_\lambda(2,1^{n-2}) = \frac{f_\lambda}{\binom n2}C(\lambda),
\end{displaymath}
where
\begin{displaymath}
  C(\lambda) = \sum_{(i,j)\in \lambda} (j-i),
\end{displaymath}
and $(i,j)\in \lambda$ is to be understood to mean that $(i,j)$ is a cell in the Young diagram of $\lambda$, or in other words, that $i$ and $j$ are positive integers with $j\leq \lambda_i$ for each $i$ indexing a part of $\lambda$. As a consequence (see Stanley \cite[Exercise~7.55]{ec2}):
\begin{theorem}
  Given a partition $\lambda$ of $n$, we have
\begin{equation}
  \label{eq:5}
  g_\lambda = \frac{f_\lambda\left(\binom n2 - C(\lambda)\right)}{2\binom n2}.
\end{equation}
  \label{theorem:chirality-gla}
  The partition $\lambda$ of $n$ is chiral if and only if $g_\lambda$ is odd.
\end{theorem}

\section{Proof of Theorem~\ref{theorem:chiral-towers}}
The $2$-core tower of a partition records how it can be reduced to its core by removing maximal rim-hooks whose sizes are powers of two.
Theorem~\ref{theorem:chiral-towers} will be proved by analyzing the contributions of the contents of these rim-hooks to $C(\lambda)$.
This is achieved in two parts: the first is a reduction to the case where $n=\epsilon + 2^k$, with $\epsilon\in \{0,1\}$ (Lemma~\ref{lemma:reduction-to-powers-of-2}), and the second is a characterization of chiral partitions of $\epsilon + 2^k$ (Lemma~\ref{lemma:powers-of-two}).
\begin{lemma}
  \label{lemma:core-tower-entries}
  Let $\lambda$ be any partition.
  For each binary sequence $x$, let $\lambda_x$ denote the partition obtained recursively from $\lambda$ by (\ref{eq:6}).
  Fix $\delta \in \{0, 1\}$.
  The nodes of $\lambda_x$, as $x$ runs over all binary sequences of length $i$ beginning with $\delta$, correspond to the nodes of $\lambda$ whose hook-lengths are multiples of $2^i$ and whose hand nodes have content congruent to $\delta$ modulo $2$.
\end{lemma}
\begin{proof}
  The lemma follows by applying \cite[Theorem~3.3]{olsson} recursively.
\end{proof}
\begin{lemma}
  \label{lemma:core-tower-of-core}
  Let $\lambda$ be a partition.
  Then the $2$-core tower of $\core\lambda{2^i}$ is obtained by replacing all the partitions in rows numbered $i$ and larger by the empty partition in the $2$-core tower of $\lambda$.
\end{lemma}
\begin{proof}
  Since the $2$-core tower of a $2$-core is concentrated in the $0$th row, the lemma holds for $i=0$.
  The proof proceeds by induction on $i$ using the following general fact:
  for integers $p$ and $q$, we always have
  \begin{equation}
    \label{eq:16}
    \quo{(\core\lambda{pq})}p = \core{(\quo\lambda p)}q,
  \end{equation}
  where the right hand side is interpreted as the $q$-tuple of cores of the partitions in $\quo\lambda p$.
  This can be deduced from the assertion about the removal of hooks in \cite[Theorem~3.3]{olsson}.

  Let $\omega$ be the partition whose $2$-core tower is obtained from that of $\lambda$ by replacing all the partitions in rows numbered $i$ and larger by the trivial partition.
  We wish to show that $\omega = \core\lambda{2^i}$.
  Clearly, $\core \lambda 2 = \core{(\core\lambda{2^i})}2 = \core \omega 2$.
  
  The $2$-core tower $T(\lambda)$ of $\lambda$ looks like
  \begin{displaymath}
    T(\lambda) = 
    \xymatrix{
      &\core{\lambda}{2} \ar@{-}[dl] \ar@{-}[dr]&\\
      T(\lambda_0) && T(\lambda_1)
    }
  \end{displaymath}
  where $\quo\lambda 2 = (\lambda_0, \lambda_1)$, and $T(\lambda_i)$ is the $2$-core tower of $\lambda_i$.
  Similarly, the $2$-core tower of $\omega$ looks like:
  \begin{displaymath}
    T(\omega) = 
    \xymatrix{
      &\core{\lambda}{2} \ar@{-}[dl] \ar@{-}[dr]&\\
      T(\omega_0) && T(\omega_1)
    }
  \end{displaymath}
  where $\quo\omega 2 = (\omega_0, \omega_1)$.
  Since $T(\omega)$ is obtained from $T(\lambda)$ by replacing all the partitions in rows numbered $i$ and larger by the empty partition, $T(\omega_i)$ is obtained from $T(\lambda_i)$ by replacing partitions in rows numbered $i-1$ and larger by the empty partition.
  By the induction hypothesis, $\omega_i = \core{\lambda_i}{2^{i-1}}$.
  We have:
  \begin{align*}
    \quo\omega 2 & = \core{(\quo\lambda 2)}{2^{i-1}}\\
    & = \quo{(\core\lambda{2^i})}2,
  \end{align*}
  by applying (\ref{eq:16}) with $p=2$ and $q=2^{i-1}$.
  We have shown that $\omega$ and $\core\lambda{2^i}$ have the same $2$-core and $2$-quotient, from which follows that $\omega = \core\lambda{2^i}$.
\end{proof}
\begin{lemma}
  \label{lemma:valuation-of-core}
  For any partition $\lambda$, and any positive integer $i$,
  \begin{displaymath}
    v_2(f_\lambda) = v_2(f_{\core\lambda{2^i}}) + v_2(f_\mu) + \nu(n-|\core\lambda{2^i}|) + \nu(|\core\lambda{2^i}|) - \nu(n). 
  \end{displaymath}
  where $\mu$ denotes the partition whose $2$-core tower is obtained from the $2$-core tower of $\lambda$ by replacing all the partitions in rows numbered $0,\dotsc,i-1$ by the empty partition.
  As a consequence,
  \begin{displaymath}
    v_2(f_\lambda) \geq v_2(f_{\core\lambda{2^i}}) + v_2(f_\mu).
  \end{displaymath}
\end{lemma}
\begin{remark*}
  In the inequality in Lemma~\ref{lemma:valuation-of-core}, it should be understood that for the empty partition $\emptyset$ of $0$, we have $f_\emptyset = 1$.
\end{remark*}
\begin{proof}
  Let $\alpha = \core\lambda{2^i}$.
  By Theorem~\ref{theorem:oddness}, we have:
  \begin{displaymath}
    v_2(f_\lambda) = \sum_{j=0}^{i-1} w_j(\lambda) + \sum_{j\geq i} w_j(\lambda) - \nu(n),
  \end{displaymath}
  and by the additional use of Lemma~\ref{lemma:core-tower-of-core},
  \begin{displaymath}
    v_2(f_\alpha) = \sum_{j=0}^{i-1} w_j(\lambda) - \nu(|\alpha|).
  \end{displaymath}
  So we have:
  \begin{align*}
    v_2(f_\lambda) - v_2(f_\alpha) & = \sum_{j\geq i} w_i(\lambda) - (\nu(n) - \nu(|\alpha|))\\
    & = \sum_{j\geq i} w_i(\lambda) - \nu(n - |\alpha|) + [\nu(n-|\alpha|) + \nu(|\alpha|) - \nu(n)]\\
    & = v_2(f_\mu) +\nu(n-|\alpha|) + \nu(|\alpha|) - \nu(n),
  \end{align*}
  completing the proof of the identity in the lemma.
  Recall that for any integers $k$ and $l$, $\nu(k+l)\leq \nu(k) + \nu(l)$.
  In particular, $\nu(n-|\core\lambda{2^i}|) + \nu(|\core\lambda{2^i}|) - \nu(n)\geq 0$, so the inequality in the lemma follows.
\end{proof}
\begin{lemma}
  \label{lemma:reduction-to-powers-of-2}
  Let $\lambda$ be a partition of a positive integer $n$ as in (\ref{eq:1}).
  Let $\alpha = \core\lambda{2^{k_1+1}}$.
  Then $\lambda$ is chiral if and only if both the following conditions hold:
  \begin{enumerate}
  \item \label{item:1} $\alpha$ is a chiral partition of $\epsilon + 2^{k_1}$.
  \item \label{item:2} If $\mu$ is the partition whose $2$-core tower is obtained from the $2$-core tower of $\lambda$ by replacing the partitions appearing in rows numbered $0,\dotsc,k_1$ by the empty partition, then $v_2(f_\mu) = 0$.
  \end{enumerate}
\end{lemma}
\begin{proof}
  Suppose $\lambda$ is chiral.
  Recall that the $2$-adic valuation of a binomial coefficient $\binom nk$ is $\nu(k) + \nu(n-k) - \nu(n)$.
  It follows that $v_2\binom n2 = k_1-1$ if $n$ is of the form (\ref{eq:1}).
  Moreover, since $\alpha$ is obtained from $\lambda$ by removing $2^{k_1+1}$-rim hooks, $|\alpha| \equiv n\mod 2^{k_1+1}$.
  So,
  \begin{equation}
    \label{eq:11}
    v_2\binom n2 = v_2\binom{|\alpha|}2 = k_1-1.
  \end{equation}
  
  As $(i,j)$ runs over a rim-hook $r$ of size $2^{k_1+1}$, $(j-i)$ visits each residue class modulo $2^{k_1+1}$ exactly once.
  So
  \begin{displaymath}
    \sum_{(i,j)\in r} (j-i) \equiv 2^{k_1}(2^{k_1+1}-1) \equiv 0 \mod 2^{k_1}.
  \end{displaymath}
  Since $\alpha$ is obtained from $\lambda$ by removing a sequence of such hooks, we have:
  \begin{equation}
    \label{eq:10}
    C(\lambda)\equiv C(\alpha) \mod 2^{k_1}.
  \end{equation}
  Since $v_2(g_\lambda) =0$, (\ref{eq:5}) gives
  \begin{displaymath}
    v_2\left(\binom n2 - C(\lambda)\right) = k_1 - v_2(f_\lambda) \leq k_1.
  \end{displaymath}
  Together with (\ref{eq:10}), we get that
  \begin{equation}
    \label{eq:12}
    v_2\left(\binom n2 - C(\lambda)\right) = v_2\left(\binom{|\alpha|}2 - C(\alpha)\right).
  \end{equation}
  Since $v_2(f_\alpha)\leq v_2(f_\lambda)$ (by Lemma~\ref{lemma:valuation-of-core}), it follows from (\ref{eq:5}) that $v_2(g_\alpha)\leq v_2(g_\lambda)$.
  Since $v_2(g_\lambda) = 0$, it must be that $v_2(g_\alpha)=0$, i.e., $\alpha$ is also chiral.

  Another consequence of the above argument is that $v_2(f_\alpha) = v_2(f_\lambda)$.
  which (by Lemma~\ref{lemma:valuation-of-core}) means that $\nu(n) = \nu(|\alpha|) + \nu(n-|\alpha|)$, and $v_2(f_\mu)=0$.
  In particular, the $2$-deviation of $\alpha$ comes from the first $k_1$-rows of its $2$-core tower, so the binary expansion of $|\alpha|$ cannot have digits beyond the $k_1$th place value.
  Since $\nu(n) = \nu(|\alpha|) + \nu(n-|\alpha|)$, we must have $|\alpha| = 2^{k_1}+\epsilon$.

  For the converse, note that the conditions (\ref{item:1}) and (\ref{item:2}) imply that $e_2(\mu)=0$ and hence $e_2(\lambda) = e_2(\alpha)$, so $v_2(f_\alpha) = v_2(f_\lambda)$.
  Also, $v_2\binom n2 = v_2\binom{2^{k_1}+\epsilon}2 = k_1-1$.
  So in order to show that $\lambda$ is chiral, it suffices to show that
  \begin{equation}
    \label{eq:ineq-val}
    v_2\left(\binom n2 - C(\lambda)\right) \leq v_2\left(\binom{2^{k_1} + \epsilon}2 - C(\alpha)\right).
  \end{equation}
  Since $\alpha$ is chiral and $v_2(g_\alpha) =0$.
  Now (\ref{eq:5}) implies that $v_2\left(\binom{2^{k_1} + \epsilon}2 - C(\alpha)\right)\leq k_1$.
  The identity (\ref{eq:10}) still holds true, so we must have equality in (\ref{eq:ineq-val}).
\end{proof}  
\begin{lemma}
  \label{lemma:powers-of-two}
  Let $\lambda$ be a partition of $n = \epsilon + 2^k$, where $\epsilon\in\{0,1\}$, and $k\geq 1$.
  Then $\lambda$ is chiral if and only if one of the following conditions holds:
  \begin{enumerate}
  \item \label{item:3} The partition $\lambda$ satisfies
    \begin{displaymath}
      w_i(\lambda) =
      \begin{cases}
        1 &\text{if } i=k, \text{ or if } \epsilon = 1\text{ and } i =0,\\
        0 &\text{otherwise,}
      \end{cases}
    \end{displaymath}
    and the unique non-empty partition in the $k$th row of the $2$-core tower of $\lambda$ is $\alpha_x$, where the binary sequence $x$ of length $k$ begins with $\epsilon$.
    In this case $f_\lambda$ is odd.
  \item \label{item:4} For some $0<v<k$,
    \begin{displaymath}
      w_i(\lambda) =
      \begin{cases}
        2 & \text{if } i = k-v,\\
        1 & \text{if } k-v+1\leq i\leq k-1,\text{ or if } \epsilon = 1 \text{ and } i = 0,\\
        0 & \text{otherwise,}
      \end{cases}
    \end{displaymath}
    and the two non-empty partitions in the $(k-v)$th row of the $2$-core tower of $\lambda$ are $\alpha_x$ and $\alpha_y$, for binary sequences $x$ and $y$ such that $x$ begins with $0$ and $y$ begins with $1$.
    In this case $v_2(f_\lambda) = v$.
  \item \label{item:5} We have $\epsilon = 1$ and the partition $\lambda$ satisfies
    \begin{displaymath}
      w_i(\lambda) =
      \begin{cases}
        3 & \text{if } i=0,\\
        1 & \text{if } 1\leq i\leq k-1.
      \end{cases}
    \end{displaymath}
    In this case, $v_2(f_\lambda) = k$.
  \end{enumerate}
\end{lemma}
\begin{proof}
  Let $v = v_2(f_\lambda)$.
  Since $v_2\binom n2 = k-1$, Theorem~\ref{theorem:chirality-gla} implies that $\lambda$ is chiral if and only if
  \begin{equation}
    \label{eq:3}
    v_2\left(\binom n2 - C(\lambda)\right) = k - v.
  \end{equation}
  Consider first the case where $v = 0$.
  Then, by Theorem~\ref{theorem:oddness}, $\lambda$ has a unique $2^k$-hook.
  Suppose that the content (column number minus row number) of the foot node of this hook is $c+\epsilon$ for some integer $c$.
  Then the nodes in the rim $r$ of this hook have contents $c+\epsilon, c+\epsilon + 1,\dotsc, c+n-1$.
  Consequently, if we write $C(r)$ for the sum of contents of the nodes in $r$, then
  \begin{displaymath}
    C(r) = 2^k c + \binom n2.
  \end{displaymath}
  Furthermore $\lambda$ has at most one node not in $r$, whose content is $0$.
  So $C(\lambda) = C(r)$ and we have
  \begin{displaymath}
    \binom n2 - C(\lambda) = -2^k c,
  \end{displaymath}
  which has valuation $k$ if and only if $c$ is odd.
  Since $r$ is a hook of even length, its hand and foot nodes always have contents of different parity.
  So $\lambda$ is chiral if and only if the content of the hand node of $r$ has the same parity as $\epsilon$.
  By Lemma~\ref{lemma:core-tower-entries} this is equivalent to saying that the unique non-trivial partition in the $k$th row of the $2$-core tower of $\lambda$ is $\lambda_x$, where $x$ is a binary sequence beginning with $\epsilon$.

  Now consider the case where $v = 1$ and $k\geq 2$.
  Then, by Theorem~\ref{theorem:oddness},
  \begin{displaymath}
    w_i(\lambda) =
    \begin{cases}
      2 &\text{if } i = k-1,\\
      1 &\text{if } \epsilon = 1 \text{ and } i = 0,\\
      0 &\text{otherwise.}
    \end{cases}
  \end{displaymath}
  Thus the $2$-core of $\lambda$, which is a partition of size $\epsilon$ is obtained by removing two rim-hooks $r_1$ and $r_2$ of size $2^{k-1}$ from $\lambda$.
  As a consequence, $C(\lambda) = C(r_1) + C(r_2)$.
  Suppose that the hand-nodes of these hooks have contents $c_1$ and $c_2$ respectively.
  Then
  \begin{displaymath}
    C(r_1) + C(r_2) = 2^{k-1}(c_1 + c_2) - 2^{k-1}(2^{k-1}-1).
  \end{displaymath}
  We have
  \begin{align*}
    \binom n2 - C(\lambda) & = 2^{k-1}(2^k - 1 + 2\epsilon) - (C(r_1) + C(r_2))\\
    & = 2^k(2^{k-1} + 2^{k-2} - 1 + \epsilon) - 2^{k-1}(c_1 + c_2).
  \end{align*}
  This has valuation $k-1$ if and only if $c_1 + c_2$ is odd, in other words, the two non-trivial entries in the $2$-core tower of $\lambda$ are in the $(k-1)$st row and are of the form $\lambda_x$ and $\lambda_y$ where $x$ and $y$ are binary sequences such that $x$ begins with $0$, which $y$ begins with $1$.

  Now consider the case where $1 < v < k$.
  Theorem~\ref{theorem:oddness} implies that $w_i(\lambda) = 0$ for all $i<k-v$, except if $\epsilon = 1$, in which case $w_0(\lambda)=1$.
  Let $j$ be the smallest positive integer for which $w_j(\lambda)>0$.
  Then $j\geq k-v$, and since $\sum_{i\geq j} 2^i w_i(\lambda) = 2^k$, $w_j(\lambda)$ must be even.
  For each $i\geq j$, $\core\lambda{2^i}$ is obtained from $\core\lambda{2^{i+1}}$ by the removal of $w_i(\lambda)$ $2^i$-rim-hooks, and $\core\lambda{2^j}$ is a partition of $\epsilon$.
  Now, if $r$ is a rim-hook of size $2^i$ for some $i> j$ with hand-node content $c$, then since its nodes have contents $c, c-1, \dotsc, c-(2^i-1)$,
  \begin{displaymath}
    C(r) = 2^ic - \binom{2^i}2 \equiv 0 \mod 2^j,
  \end{displaymath}
  so $C(\lambda)\equiv C(\core\lambda{2^i}) \mod 2^j$ for all $i>j$.
  If $r_1$ and $r_2$ are two rim-hooks of size $2^j$, with hand-node contents $c_1$ and $c_2$, then
  \begin{displaymath}
    C(r_1) + C(r_2) = 2^j(c_1 + c_2) + 2^j(2^j-1)\equiv 0 \mod 2^j,
  \end{displaymath}
  So $C(\lambda)\equiv C(\core\lambda{2^j}) \mod 2^j$.
  But $\core\lambda{2^j} = \core\lambda 2$ is a partition of $\epsilon$, so $C(\core\lambda{2^j}) = 0$.
  Thus, if $j>k-v$, then
  \begin{displaymath}
    v_2\left(\binom n2 - C(\lambda)\right) > k-v.
  \end{displaymath}
  It follows that if $\lambda$ is chiral, then $j = k-v$, whence the conditions that $w_{k-v}(\lambda) = 2$, and $w_i(\lambda)=1$ for $k-v<i<k$ are forced, and (\ref{eq:3}) holds if and only if the hand nodes of the two $2^{k-v}$-hooks have contents of opposite parity.

  Now consider the case where $v = k$.
  In this case $\lambda$ is chiral if and only if $C(\lambda)$ is odd.
  If $r$ is a rim-hook of size $2^i$, then $C(r)$ is even unless $i=1$.
  So $C(\lambda)\equiv C(\core\lambda 4) \mod 2$.
  If $r$ is a rim-hook of size $2$, then $C(r)$ is odd.
  Also, $C(\core\lambda 2) = 0$.
  It follows that $\lambda$ is chiral if and only if $w_1(\lambda)$ is odd.
  Theorem~\ref{theorem:oddness} now implies that $w_i(\lambda) =1$ for all $1\leq i<k$, and that the $2$-core of $\lambda$ is a partition of $2+\epsilon$.
  Since no partition of $2$ is a $2$-core, this case can occur only when $\epsilon = 1$, giving the case (\ref{item:5}).

  This leaves us with the case where $v>k$.
  Since the denominator of the right hand side of (\ref{eq:5}) has valuation $k$, $g_\lambda$ can never be odd in this case.
\end{proof}
Reading Lemma~\ref{lemma:reduction-to-powers-of-2} together with Lemma~\ref{lemma:powers-of-two} now gives Theorem~\ref{theorem:chiral-towers}.
\section{The Growth of $b(n)$ and Sage code}
\label{sec:growth-bn}
Recall that, if $n$ is a positive integer with binary expansion (\ref{eq:1}),
then the number of partitions of $n$ such that $f_\lambda$ is odd (which we will call \emph{odd partitions}) is given by
\begin{equation}
  \label{eq:17}
  a(n) = 2^{k_1+\dotsb + k_r}.
\end{equation}
This result was proved by Macdonald \cite{macdonald1971degrees}.
It has a nice interpretation in terms of Theorem~\ref{theorem:oddness}:
If $n$ has binary expansion (\ref{eq:1}) and $\lambda$ is a partition of $n$ such that $f_\lambda$ is odd, then the rows $k_1,\dotsc,k_r$ of the $2$-core tower of $\lambda$ have one non-empty entry each, namely the partition $(1)$.
There are $2^{k_i}$ choices for the location of the entry in the $k_i$th row.
The remaining rows, with the exception of the $0$th row when $\epsilon =1$, have all entries $\emptyset$, and (\ref{eq:17}) is obtained.
The subgraph of Young's graph consisting of odd partitions is discussed in \cite{2016arXiv160101776A}.

It turns out that the functions $a(n)$ and $b(n)$ track each other closely (see Figure~\ref{fig:growth}).
\begin{figure}[h]
  \centering
  \includegraphics[scale=0.6]{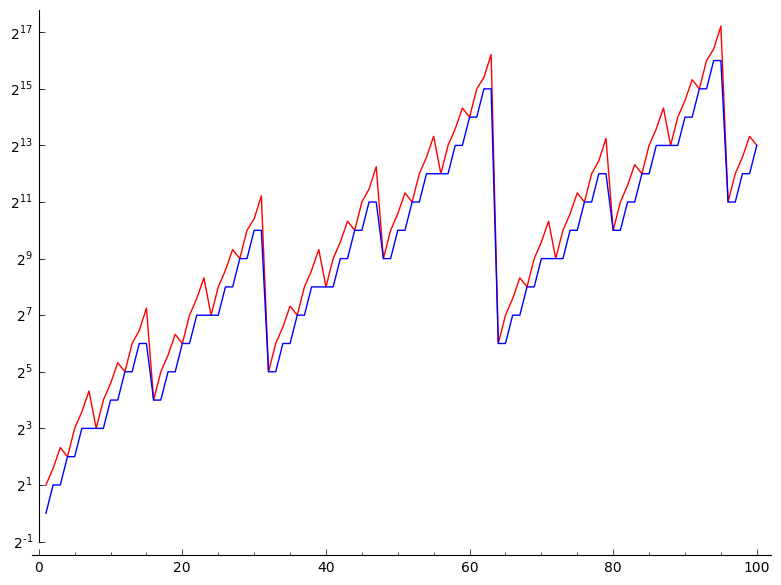}
  \caption{Growth of \textcolor{blue}{$a(n)$} (odd partitions) against \textcolor{red}{$b(n+2)$} (chiral partitions)  on a logarithmic scale}
  \label{fig:growth}
\end{figure}
\begin{theorem}
  \label{theorem-an-bn+2-comparison}
  For every positive integer $n$,
  \begin{displaymath}
    2/5\leq a(n)/b(n+2) \leq 1.
  \end{displaymath}
  Moreover, $a(n)/b(n+2) = 1$ if and only if $n$ is divisible by $4$.
\end{theorem}
\begin{proof}
  Suppose that $n+2$ has the form (\ref{eq:1}), then the formula for $b(n)$ in Theorem~\ref{theorem:all} implies that
  \begin{displaymath}
    \frac{a(n)}{b(n+2)} = \frac{a(2^{k_1} + \epsilon)}{b(2^{k_1} + \epsilon + 2)}.
  \end{displaymath}
  Therefore it suffices to prove the inequalities for $n+2 = 2^k+\epsilon$ for $k\geq 1$ and $\epsilon\in \{0,1\}$.

  We will first prove the lower bound on $a(n)/b(n+2)$, which can be rewritten as:
  \begin{equation}
    \label{eq:7}
    2b(n+2)\leq 5a(n)
  \end{equation}

  We will first show that
  \begin{equation}
    \label{eq:8}
    2b(2^k) \leq 3a(2^k-2).
  \end{equation}
  The left hand side comes from counting, as $v$ goes from $0$ to $k-1$, the possible $2$-core towers that give rise to partitions $\lambda$ of $2^k$ with $v_2(f_\lambda) = v$, and $v_2(g_\lambda) = 0$.
  For $0<v\leq k-1$, this entails the choice of a sequence of entries in rows $v+1,\dots,k-1$, and the choice of two entries in the $v$th row, one in each half.
  Given such a configuration $C$, define $\phi(C)$ to be the configuration obtained by replacing the right entry of the $(k-v)$th row of $C$ by the entries in rows $1,\dotsc,k-v-1$ which lie on the geodesic joining this entry to the apex of the tower.

  For example, taking $k=5$ and $v=2$, Figure~\ref{fig:c} shows a possible configuration of non-trivial entries (represented by colored circles) in the $2$-core tower of a chiral partition $\lambda$ of $32$ with $v(f_\lambda) = 2$.
  \begin{figure}[h]
    \centering
    \includegraphics[scale = 0.5]{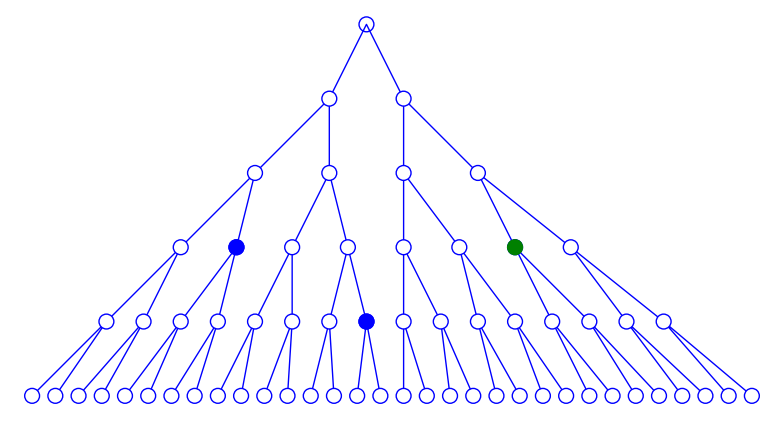}
    \caption{The configuration $C$}
    \label{fig:c}
  \end{figure}
  The configuration $\phi(C)$ is obtained by replacing the green entry in this figure by the two green entries in Figure~\ref{fig:phic}.
  \begin{figure}[h]
    \centering
    \includegraphics[scale = 0.5]{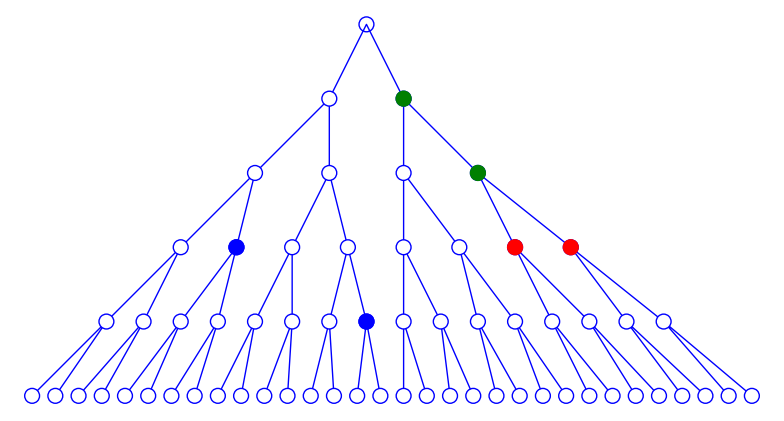}
    \caption{The configuration $\phi(C)$}
    \label{fig:phic}
  \end{figure}
  The configuration $\phi(C)$ comes from exactly two possible configurations $C$.
  The possible positions of the right side entry in the $(k-v)$th row of these configurations are indicated in red.
  In general, they are the left child and right child of the entry in the row just above the first row with an entry in the left half of $\phi(C)$.
  Thus $C\mapsto \phi(C)$ is a $2$-to-$1$ map from a set which counts $\sum_{v=1}^{k-1}b_v(2^k)$ to a set which counts $a(2^k-2)$.
  Thus we have
  \begin{equation}
    \label{eq:9}
    2\sum_{v=1}^{k-1} b_v(2^k) \leq a(2^k-2).
  \end{equation}
  Thus we will have proved (\ref{eq:8}) if we show that $b_0(2^k)\leq a(2^k-2)$.
  But this is clear, since
  \begin{displaymath}
    b_0(2^k) = 2^{k-1} \leq 2^{\binom k2} = a(2^k-2) \text{ for all } k\geq 1.
  \end{displaymath}

  Note that $b(2^k+1) = b(2^k) + a(2^k-1)$, and that $a(2^k-2) = a(2^k-1)$.
  So we have
  \begin{align*}
    2b(2^k+1) & = 2b(2^k) + 2a(2^k-1)\\
    & \leq 3a(2^k-2) + 2a(2^k-1) & \text{[using (\ref{eq:8})]}\\
    & = 5a(2^k-2).
  \end{align*}
  Together with (\ref{eq:8}), this establishes (\ref{eq:7}) for $n=2^k + \epsilon$.

  The upper bound on $a(n)/b(n+2)$ says that $a(n)\leq b(n+2)$.
  Once again, it may be assumed that $n+2 = 2^k+\epsilon$.
  For $k<3$, the inequality is verified by direct computation.
  For $k\geq 3$, we will show that
  \begin{displaymath}
    b_{k-1}(2^k+\epsilon) + b_{k-2}(2^k+\epsilon) = a(2^k+\epsilon -2).
  \end{displaymath}
  By Theorem~\ref{theorem:refined}, (and dividing by the right hand side) this reduces to:
  \begin{equation}
    \label{eq:18}
    2^{k(k-2) - \binom{k-1}2 - \binom k2} + 2^{(k-1)(k-2) - \binom{k-2}2 - \binom k2} = 1.
  \end{equation}
  The exponents of the left hand side both simplify to $-1$, and the identity follows.
\end{proof}
Theorem \ref{theorem-an-bn+2-comparison} says that $a(n)$ is a good proxy for estimating the growth of $b(n)$.
The order of the sequence $a(n)$ fluctuates widely; when $n=2^k$, $a(n) = 2^k$ and when $n=2^k-1$, $a(n) = 2^{k(k-1)/2}$.
In any case, $a(n)$ is dwarfed by the growth of the partition function:
\begin{displaymath}
  p(n) \sim \frac 1{4n\sqrt 3} \exp(\pi \sqrt{2n/3}) \text{ as } n\to \infty.
\end{displaymath}
For example, Theorem~\ref{theorem:all} predicts a relatively large value for $b(4097)$ (compared to neighboring integers), but even so, the probability of a partition of $4097$ being chiral is
\begin{displaymath}
  b(4097)/p(4097) \approx 4.488811279418092\times 10^{-30},
\end{displaymath}
which is astronomically small.
But using Theorem~\ref{theorem:chiral-towers}, one may easily generate a random chiral partition of $4097$.
A sample run of our code, which is available at \url{http://www.imsc.res.in/\~amri/chiral.sage}, gives a random chiral partition of $4097$ instantaneously on a conventional office desktop:
\begin{lstlisting}
sage: random_chiral_partition(4097).frobenius_coordinates()
([1879, 272, 152, 27, 20, 19, 8, 2, 0],
 [1015, 239, 168, 103, 100, 43, 32, 7, 2])
\end{lstlisting}
Our code based on Theorem~\ref{theorem:chiral-towers} also provides functions for generating random chiral partitions with dimension having fixed $2$-adic valuation, and for enumerating all chiral partitions of $\lambda$ of $n$ (with dimension having fixed $2$-adic valuation if desired).

\section{Chiral Hooks}
\label{sec:chiral-hooks}

Hooks, namely partitions of the form $h(a,b) = (a+1,1^b)$ for $a,b\geq 0$, form an order ideal in Young's lattice.
For a partition $\lambda$ of $n$, let $\lambda^-$ denote the set of partitions of $n-1$ whose Young diagrams are obtained by removing a box from the Young diagram of $\lambda$.
The statistic $f_\lambda$ is given by the recursive rule
\begin{displaymath}
  f_\lambda = \sum_{\mu\in \lambda^-} f_\mu,
\end{displaymath}
with the initial condition $f_{(1)} = 1$.
For hooks, this can be rewritten as:
\begin{displaymath}
  f_{h(a,b)} = f_{h(a-1, b)} + f_{h(a, b-1)},
\end{displaymath}
with the initial condition $f_{h(0,0)} = 0$.
This is nothing but the ancient rule of the Sanskrit prosodist Pingala for generating the \emph{Meru Prastaara} (now known as Pascal's triangle; see \cite{shah-pingala}).
It is well-known (see Theorem~\ref{theorem:odd-multi}) that the number of odd entries in the $n$th row of the Meru Prastaara (i.e., the number of integers $0\leq k\leq n$ for which $\binom nk$ is odd) is $2^{\nu(n)}$.

Chirality depends on the parity of $g_\lambda$ which, like $f_\lambda$, satisfies the recursive identity
\begin{displaymath}
  g_\lambda = \sum_{\mu\in \lambda^-} g_\mu,
\end{displaymath}
but only for $|\lambda|>2$, and with initial conditions $g_{(2)} = 0$ and $g_{(1,1)} = 1$.
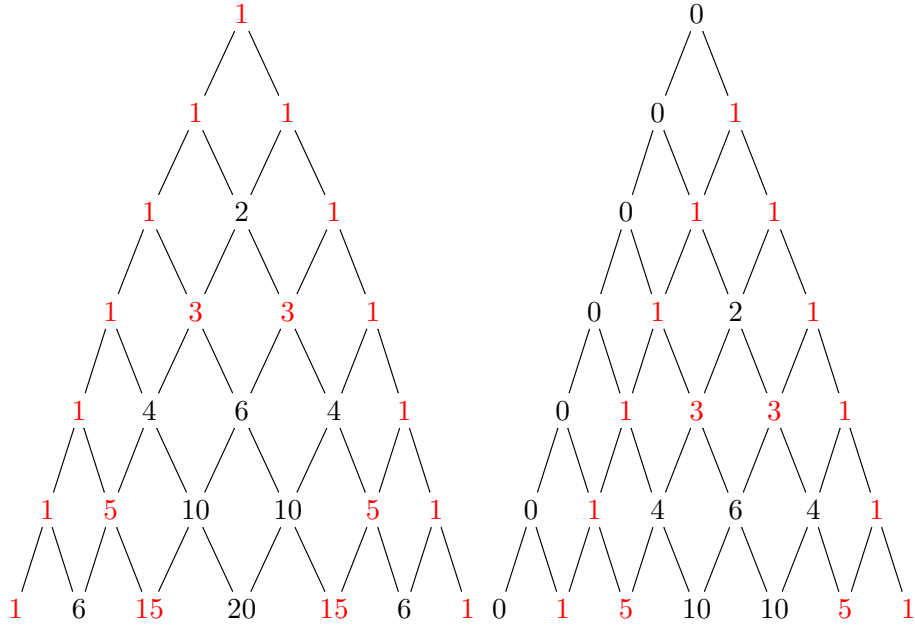
\begin{figure}[h]
  \begin{displaymath}
    \xymatrix@C=0em{
      &&&&&&\textcolor{red}{1}\ar@{-}[dl]\ar@{-}[dr]&&&&&&\\
      &&&&&\textcolor{red}{1}\ar@{-}[dl]\ar@{-}[dr]&&\textcolor{red}{1}\ar@{-}[dl]\ar@{-}[dr]&&&&&\\
      &&&&\textcolor{red}{1}\ar@{-}[dl]\ar@{-}[dr]&&2\ar@{-}[dl]\ar@{-}[dr]&&\textcolor{red}{1}\ar@{-}[dl]\ar@{-}[dr]&&&&\\
      &&& \textcolor{red}{1}\ar@{-}[dl]\ar@{-}[dr] && \textcolor{red}{3}\ar@{-}[dl]\ar@{-}[dr] && \textcolor{red}{3}\ar@{-}[dl]\ar@{-}[dr] && \textcolor{red}{1}\ar@{-}[dl]\ar@{-}[dr] &&&\\
      && \textcolor{red}{1}\ar@{-}[dl]\ar@{-}[dr] && 4\ar@{-}[dl]\ar@{-}[dr] && 6\ar@{-}[dl]\ar@{-}[dr] && 4\ar@{-}[dl]\ar@{-}[dr] && \textcolor{red}{1}\ar@{-}[dl]\ar@{-}[dr] &&\\
      &\textcolor{red}{1}\ar@{-}[dl]\ar@{-}[dr] && \textcolor{red}{5}\ar@{-}[dl]\ar@{-}[dr] && 10\ar@{-}[dl]\ar@{-}[dr] && 10\ar@{-}[dl]\ar@{-}[dr] && \textcolor{red}{5}\ar@{-}[dl]\ar@{-}[dr] && \textcolor{red}{1}\ar@{-}[dl]\ar@{-}[dr]\\
      \textcolor{red}{1} && 6 && \textcolor{red}{15} && 20 && \textcolor{red}{15} && 6 && \textcolor{red}{1}
    }
    \xymatrix@C=0em{
      &&&&&&0\ar@{-}[dl]\ar@{-}[dr]&&&&&&\\
      &&&&&0\ar@{-}[dl]\ar@{-}[dr]&&\textcolor{red}{1}\ar@{-}[dl]\ar@{-}[dr]&&&&&\\
      &&&&0\ar@{-}[dl]\ar@{-}[dr]&&\textcolor{red}{1}\ar@{-}[dl]\ar@{-}[dr]&&\textcolor{red}{1}\ar@{-}[dl]\ar@{-}[dr]&&&&\\
      &&& 0\ar@{-}[dl]\ar@{-}[dr] && \textcolor{red}{1}\ar@{-}[dl]\ar@{-}[dr] && 2\ar@{-}[dl]\ar@{-}[dr] && \textcolor{red}{1}\ar@{-}[dl]\ar@{-}[dr] &&&\\
      && 0\ar@{-}[dl]\ar@{-}[dr] && \textcolor{red}{1}\ar@{-}[dl]\ar@{-}[dr] && \textcolor{red}{3}\ar@{-}[dl]\ar@{-}[dr] && \textcolor{red}{3}\ar@{-}[dl]\ar@{-}[dr] && \textcolor{red}{1}\ar@{-}[dl]\ar@{-}[dr] &&\\
      &0\ar@{-}[dl]\ar@{-}[dr] && \textcolor{red}{1}\ar@{-}[dl]\ar@{-}[dr] && 4\ar@{-}[dl]\ar@{-}[dr] && 6\ar@{-}[dl]\ar@{-}[dr] && 4\ar@{-}[dl]\ar@{-}[dr] && \textcolor{red}{1}\ar@{-}[dl]\ar@{-}[dr]\\
      0 && \textcolor{red}{1} && \textcolor{red}{5} && 10 && 10 && \textcolor{red}{5} && \textcolor{red}{1}
    }
  \end{displaymath}
  \caption{The numbers $f_\lambda$ (left) and $g_\lambda$ (right) for hooks.}
  \label{fig:Merus}
\end{figure}
When restricted to hooks, this results in a shifted version (see Figure~\ref{fig:Merus}) of the Meru Prastaara.
We have:
\begin{displaymath}
  g_{h(a, b)} =
  \begin{cases}
    f_{h(a, b-1)} & \text{if } b > 0,\\
    0 & \text{otherwise}.
  \end{cases}
\end{displaymath}
An immediate consequence is:
\begin{theorem}
  The hook $h(a,b)$ is chiral if and only if $b>0$ and $\binom{a+b-1}{a}$ is odd.
  Thus the number of chiral hooks of size $n$ is $2^{\nu(n-1)}$.
\end{theorem}

\section{Chirality for Permutation Representations}
\label{sec:perm-reps}

\subsection{Characterization of Chiral Permutation Representations}

For a positive integer
\begin{equation}
  \label{eq:13}
  n = 2^{k_0} + \dotsb + 2^{k_r},\, 0\leq k_0<\dotsb<k_r,
\end{equation}
define $\bin(n)=\{k_0,\ldots,k_r\}$, a set of cardinality $\nu(n)$.
We say that a partition $\lambda = (\lambda_1,\dotsc,\lambda_l)$ of $n$ is \emph{neat} when $\bin(n)$ is the disjoint union of $\bin(\lambda_1),\dotsc,\bin(\lambda_l)$.
Thus, the number of neat partitions of $n$ is the Bell number $B_{\nu(n)}$.
Given $\lambda$ and a part $\la_i$ of $\la$ with $\bin(\la_i) \subseteq \bin(n)$, note that $\la$ is neat if and only if the partition obtained by removing $\la_i$ from $\la$ is neat.
\begin{theorem}  \label{perm.chir.descr} (Characterization) Let $\la$ be a partition of $n$.  Then  $\C[X_\lambda]$ is chiral if and only if one of the following holds:
\begin{enumerate}
\item  $\lambda$ has exactly two or three odd parts, and the partition obtained by removing these odd parts from $\la$ is neat.
\item  $\lambda$ has exactly one odd part $a$, $\bin(a)\subseteq \bin(n-2)$, and the partition obtained by removing $a$ from $\la$ is neat. 
  \end{enumerate}
  \end{theorem}

The proof of Theorem \ref{perm.chir.descr} begins with the following lemma:
\begin{lemma}
  \label{lemma:reduce-two-parts}
  For a partition $\lambda = (\lambda_1,\dotsc,\lambda_l)$ of $n$, the representation $\C[X_\lambda]$ is chiral if and only if
  \begin{displaymath} \label{summ}
    \sum_{1\leq i<j\leq l} \binom{n-2}{\lambda_1, \dotsc, \lambda_i - 1, \dotsc, \lambda_j -1, \dotsc, \lambda_l}
    \text{ is odd.}
  \end{displaymath}
\end{lemma}
\begin{proof}
  Consider the action of $s_1 \in S_n$ on $X_\lambda$.
  Its orbits in $X_\lambda$ are of cardinality $1$ or $2$.
  The representation $\C[X_\lambda]$ is chiral if and only if the number of orbits of size $2$ is odd.
  The orbits of size two are formed by elements of $X_\lambda$ where $1$ and $2$ occur in different parts of the partition.
  Such an orbit is determined by a set-partition of $\{3, 4, \dotsc, n\}$ obtained by removing $1$ and $2$ from the set-partitions in the orbits.
  This gives the formula in the lemma.
\end{proof}

The following result is well-known (see for example,  {\cite[Theorem~1]{fine}}), and tacitly used throughout this section:

\begin{theorem}
  \label{theorem:odd-multi}
  For a partition $\lambda$ of $n$, the multinomial coefficient
  \begin{displaymath}
    \binom n{\lambda_1,\dotsc,\lambda_l}
  \end{displaymath}
  is odd if and only if $\lambda$ is neat.  
\end{theorem}

{\bf Remark:}  This implies that the number of partitions $\la$ of $n$ so that the dimension of $\C[X_\lambda]$ is odd is equal to $B_{\nu(n)}$.

\begin{lemma}
  \label{lemma:many-odd-parts}
  If the partition $\lambda = (\lambda_1,\dotsc,\lambda_l)$ has either no odd parts or at least four odd parts, then $\C[X_\lambda]$ is not chiral.
  Consequently, if $n$ is even and $\C[X_\lambda]$ is chiral, then $\lambda$ has exactly two odd parts.
  If $n$ is odd and $\C[X_\lambda]$ is chiral, then $\lambda$ has exactly one or three odd parts.
\end{lemma}
\begin{proof}
  By Theorem~\ref{theorem:odd-multi}, each of the summands in Lemma~\ref{lemma:reduce-two-parts} will be even if either of the hypotheses is true, because at least two of the integers in the partitions obtained by reducing two parts of $\lambda$ by $1$ will be odd.
\end{proof}

Here is an ingredient for the case that $\la$ has exactly three odd parts; its proof is immediate from Theorem \ref{theorem:odd-multi}.
\begin{lemma} \label{ingredien}
Let $n_1+ \cdots+ n_m=n$.  If $n_1$ is odd and $n_2, \ldots, n_m$ are all even, then
    \begin{displaymath}
      \binom n{n_1,\dotsc,n_m} \equiv \binom{n-1}{n_1-1,n_2,\dotsc,n_m} \mod 2.
    \end{displaymath}
\end{lemma}

\begin{lemma}
  \label{lemma:one-odd-part}
  Suppose $\lambda$ has exactly one odd part $a$.   Then $\C[X_\lambda]$ is chiral if and only if both of the following hold:
  \begin{itemize}
  \item $\bin(a)\subseteq \bin(n-2)$.
  \item The partition obtained by removing $a$ from $\la$ is neat. 
  \end{itemize}
\end{lemma}

\begin{proof}
Write $(\lambda_1,\ldots, \lambda_m)$ for the partition obtained by removing $a$ from $\la$. 
Elementary properties of multinomial coefficients give
\begin{displaymath}
\begin{split}
 \sum_{i=1}^m \binom{n-2}{a-1,\lambda_1, \ldots, \lambda_i-1, \ldots, \lambda_m} &= \binom{n-2}{a-1} \sum_{i=1}^m \binom{n-a-1}{\lambda_1, \ldots, \lambda_i-1, \ldots, \lambda_m} \\
 &=    \binom{n-2}{a-1} \binom{n-a}{\lambda_1, \ldots, \lambda_m}.
 \end{split}
 \end{displaymath}
  The result follows, since the other terms from Lemma \ref{lemma:reduce-two-parts} are necessarily even.
\end{proof}

We have all we need to prove the characterization:
\begin{proof} [Proof of Theorem \ref{perm.chir.descr}]
Lemma \ref{lemma:many-odd-parts} breaks the problem into the cases of $1,2$, or $3$ odd parts.  The case of one odd part is Lemma \ref{lemma:one-odd-part}.  The case of two odd parts is straightforward using Lemma \ref{lemma:reduce-two-parts}.  For the case of three odd parts, use Lemma  \ref{ingredien}. 
\end{proof}

\subsection{Counting of Chiral Permutation Representations}

\begin{lemma} \label{count2} Let $n \geq 2$ be even.  The number of partitions $\la$ of $n$ with exactly two odd parts and with $\C[X_\la]$ chiral, is equal to
\begin{displaymath}
  (B_{\nu(n-2)+2} - B_{\nu(n-2) + 1} + B_{\nu(n-2)})/2. 
\end{displaymath}
\end{lemma}

\begin{proof} We must count the number of partitions $\lambda$ of $n$ which have exactly two odd parts, and so that the partition obtained by removing these odd parts from $\la$ is neat.  Equivalently, the partition of $n-2$ obtained by subtracting $1$ from these odd parts is neat.
Thus we need to count the number of set-partitions of $\bin(n-2)$ with two marked parts, these two parts being allowed to be empty (adding $1$ to these parts will give the desired partition of $n$).  To do this, add two new elements, say $a$ and $b$, to the set $\bin(n-2)$.
This set has $B_{\nu(n-2)+2}$  set-partitions, but $B_{\nu(n-1)+1}$ of these have $a$ and $b$ in the same part, and should be excluded from counting.
Except when $a$ and $b$ both lie in singleton parts (and there are $B_{\nu(n-2)}$ such partitions), each of these set partitions of $\bin(n-2)\sqcup \{a,b\}$ give rise to two set-partitions of $\bin(n-2)$ with two marked parts.
So the number of set-partitions of $\bin(n-2)$ with two marked parts, where the marked parts are not both empty, is given by
\begin{displaymath}
  (B_{\nu(n-2)+2} - B_{\nu(n-2) + 1} - B_{\nu(n-2)})/2.
\end{displaymath}

There remain the set partitions of $\bin(n-2)$ with two marked parts where both marked parts are empty.
By simply ignoring the marked parts, these are in bijective correspondence with the $B_{\nu(n-2)}$ set-partitions of $\bin(n-2)$.
The lemma follows.
\end{proof}

\begin{lemma} \label{count3} Let $n \geq 3$ be odd.  The number of partitions $\la$ of $n$ with exactly three odd parts and with $\C[X_\la]$ chiral, is equal to
\begin{displaymath}
  \frac{1}{6}(B_{\nu(n-3)+3} - 3B_{\nu(n-3)+2} + 5B_{\nu(n-3)+1} + 2B_{\nu(n-3)}).
\end{displaymath}
\end{lemma}

\begin{proof} This is similar to the previous calculation.  Such partitions of $n$ correspond to neat partitions of $n-3$ with three distinguished parts.  Once again, adding three elements to $\bin(n-3)$ and making the necessary corrections gives the result.
\end{proof}

The part of $c(n)$ corresponding to partitions with exactly one odd part is surprisingly difficult to compute, although the answer is simple.

\begin{theorem} \label{surprisingly} \label{count1} Let $n \geq 1$ be odd.  The number of partitions $\la$ of $n$ with exactly one odd part and with $\C[X_\la]$ chiral, is equal to
\begin{displaymath}
B_{\nu(n)+k-2}+B_{\nu(n)}-2B_{\nu(n)-1},
\end{displaymath}
\end{theorem}
 
\begin{proof}
The number of partitions satisfying the conditions of Lemma~\ref{lemma:one-odd-part} is
\begin{multline}
  \label{eq:19}
  \sum_{r = 0}^{\nu(n) - 2}\sum_{s = 1}^{k-1}\sum_{t = 0}^{k - s- 1}\binom{\nu(n) -2}{r}\binom{k - s - 1}{t}B_{\nu(n) + k - r - s - t - 2} \\
  + \sum_{r = 0}^{\nu(n)-2} \binom{\nu(n) - 2}{r}B_{\nu(n) - 1 - r}.
\end{multline}
where $k = v_2(n - 1)$, and $\nu(n)$ is the number of $1$'s in the binary expansion of $n$.

The reason is as follows: if $v_2(n-1) = k$, then the binary expansions of $n$ and $n-2$ are of the form:
\begin{displaymath}
  \begin{matrix}
    n = &* & * & \dotsc & * & 1 & 0 & \dotsc & 0 & 1\\
     & &   &        &   & \uparrow & & & & \\
     & &   &        &   & k & &&&\\
     & &   &        &   & \downarrow  & &&&\\
     n -2 = & * & * & \dotsc & * & 0 & 1 & \dotsc & 1 & 1\\
    \end{matrix}
\end{displaymath}
The odd part $a$ is obtained by choosing $\bin(a)\subseteq \bin(n-2)$.
Suppose that $\bin(a)$ has an element $1<i\leq k-1$.
Let $s$ be the least such $i$.
Suppose $|\bin(a) \cap \{k + 1, k + 2, \dots\}| = r$ and $|\bin(a) \cap \{s+1, \dots, k-1\}| = t$.
The number of ways of choosing these subsets  is $\binom{\nu(n) - 2}{r} \times \binom{k - s - 1}{r}$.
Once these are chosen, $|\bin(n-a)\cap \{k+1,k + 2,\dotsc\}| = \nu(n) - 2 - r$.
Also, $|\bin(n-a)\cap \{s + 1,\dots, k-1\}| = k -s - 1 - t$.
So
\begin{displaymath}
  |\bin(n-a)| = [\nu(n) - 2 - r] + [k - s - 1 - t] + 1,
\end{displaymath}
giving rise to the first term.

The second term corresponds to those values of $a$ for which $\bin(a)\cap \{2,\dots, k-1\} = \emptyset$.
Let $r = |\bin(a)\cap \{k+1,k+2,\dotsc\}|$.
The number of choices for $\bin(a)$ is then $\binom{\nu(n) - 2}{r}$.  This concludes our explanation for the expression (\ref {eq:19}).

For a sequence $f_n$, recall from finite difference calculus the $j$th differential of $f_n$, a sequence defined by:
\begin{displaymath}
(\Delta^j f)_n = \sum_{i=0}^j (-1)^i \, \binom ji f_{n+j-i}.
\end{displaymath}

One recovers $f$ from its differentials via ``Taylor's Formula"
\begin{displaymath}
f_{m+n}=\sum_{i=0}^n \binom{n}{i}(\Delta^i f)_m.
\end{displaymath}

\begin{lemma}
\label{lemma:Delta}
For $j \geq 0$, we have
\begin{displaymath}
\sum_{i=0}^n \binom{n}{i}B_{n+j-i} =(\Delta^j B)_{n+1}.
\end{displaymath}
\end{lemma}
 
\begin{proof}
Start with the explicit expression for the right hand side,
\[
\sum_{i=0}^j (-1)^i \binom ji B_{n+1+j-i}.
\]
Using the fundamental recursion for the Bell numbers  
in the above formula, we get
\[
\sum_{i=0}^j (-1)^i \binom ji \sum_{k=0}^{n+j-i} \binom{n+j-i}k B_{n+j-i-k}.
\]
We now replace the $k$ sum by $\ell = k+i$ and interchange the $\ell$ and $i$ sum to get
\[
\left( \sum_{\ell=0}^j \sum_{i=0}^\ell + \sum_{\ell=j+1}^{j+n} \sum_{i=0}^j \right)
(-1)^i \binom ji \binom{n+j-i}{\ell-i}  B_{n+j-\ell}.
\]
An application of the well-known Chu-Vandermonde identity,
\[
\sum_{i=0}^{\min(j,\ell)} (-1)^i \binom ji \binom{n+j-i}{\ell-i} = \binom n\ell,
\]
collapses the two sums above and gives the desired result.
\end{proof}
 
\begin{lemma} We have
\begin{multline*}  
 \sum_{r = 0}^{\nu(n) - 2}\sum_{s = 1}^{k-1}\sum_{t = 0}^{k - s- 1}\binom{\nu(n) -2}{r}\binom{k - s - 1}{t}B_{\nu(n) + k - r - s - t - 2} \\ =B_{\nu(n)+k-2}-B_{\nu(n)-1}.
\end{multline*}
\end{lemma}

\begin{proof}
Take the $r$-sum inside and apply Lemma~\ref{lemma:Delta} to obtain
\begin{displaymath}
\begin{split}
\sum_{s=1}^{k-1} \sum_{t=0}^{k-s-1} \binom{k-s-1}{t} (\Delta^{k-s-t}B)_{\nu(n)-1} &= \sum_{s=0}^{k-2} \sum_{t=0}^s \binom{s}{t} (\Delta^{t+1}B)_{\nu(n)-1} \\
&= \sum_{t=0}^{k-2} (\Delta^{t+1}B)_{\nu(n)-1} \sum_{s=t}^{k-2} \binom{s}{t} \\
&= \sum_{t=0}^{k-2} \binom{k-1}{t+1} (\Delta^{t+1}B)_{\nu(n)-1} \\
&= \sum_{t=1}^{k-1} \binom{k-1}{t} (\Delta^{t}B)_{\nu(n)-1} \\
&= B_{\nu(n)+k-2}-B_{\nu(n)-1}.
\end{split}
\end{displaymath}
as claimed.  (We have used Lemma \ref{lemma:Delta}  again in the last equation.)
\end{proof}
 
The second sum in (\ref{eq:19}) gives, again using Lemma~\ref{lemma:Delta}, $B_{\nu(n)}-B_{\nu(n)-1}$ (the boundary term).  Finally  for (\ref{eq:19}), we obtain
\begin{displaymath}
B_{\nu(n)+k-2}+B_{\nu(n)}-2B_{\nu(n)-1},
\end{displaymath}
completing the proof of Theorem \ref{surprisingly}.
\end{proof}

Theorem \ref{theorem:perm-rep} now follows from  Theorem \ref{perm.chir.descr} and Lemmas  \ref{count2}, \ref{count3}, and \ref{count1}.

\subsection*{Acknowledgements}
We thank Dipendra Prasad for encouraging us to work on this problem.
We thank A.~Raghuram for suggesting a preliminary version of Corollary~\ref{corollary:self-conj}.
This research was driven by computer exploration using the open-source
mathematical software \texttt{Sage}~\cite{sage} and its algebraic
combinatorics features developed by the \texttt{Sage-Combinat}
community~\cite{Sage-Combinat}. AA was partially supported by UGC centre for Advanced Study grant. AP was partially supported by a Swarnajayanti Fellowship of the Department of Science \& Technology (India).
\bibliographystyle{abbrv}
\bibliography{refs}
\end{document}